\documentclass{amsart}

\usepackage{graphicx} 
\usepackage{appendix}
\usepackage{bm}
\usepackage{braket}
\usepackage{enumerate}
\usepackage{overpic}

\newtheorem{theorem}{Theorem}[section]
\newtheorem{corollary}[theorem]{Corollary}
\newtheorem{lemma}[theorem]{Lemma}
\newtheorem{proposition}[theorem]{Proposition}

\theoremstyle{definition}
\newtheorem{example}[theorem]{Example}

\theoremstyle{remark}
\newtheorem*{remark}{Remark}

\newcommand{\sign}{\mathop{\mathrm{sign}}\nolimits}
\newcommand{\lk}{\mathop{\mathrm{lk}}\nolimits}
\renewcommand{\L}{\mathbb{L}}

\title{Purely cosmetic surgeries and Casson--Walker--Lescop invariants}

\author{Kazuhiro Ichihara}
\address{Department of Mathematics, College of Humanities and Sciences, Nihon University, 3-25-40 Sakurajosui, Setagaya-ku, Tokyo 156-8550, JAPAN}
\email{ichihara.kazuhiro@nihon-u.ac.jp}

\author{In Dae Jong}
\address{Department of Mathematics, Kindai University, 3-4-1 Kowakae, Higashiosaka City, Osaka 577-0818, Japan} 
\email{jong@math.kindai.ac.jp}

\author{Yasuyoshi Tsutsumi}
\address{Department of Education, Faculty of Education, Kobe Shinwa University, 7-13-1 Suzurandai Kitamachi, Kobe 651-1111, Japan}
\email{y-tsutsumi@ecip.kobe-shinwa.ac.jp}

\date{\today}



\thanks{This work was supported by JSPS KAKENHI Grant Numbers JP22K03301 and JP22K03324.}

\begin{document}

\begin{abstract}
Using the rational surgery formula for the Casson--Walker--Lescop invariant of links in the $3$-sphere, we show that any null-homologous knot in a rational homology sphere admits at most two pairs of integral purely cosmetic surgeries. 
We also present constraints for null-homologous knots in certain $3$-manifolds with the first Betti number one or two to admit purely cosmetic surgeries. 
As another application, we show that, for a null-homologous knot in some $3$-manifolds, including $S^2 \times S^1$, there are at most two knots which are inequivalent to the given one, but whose exteriors are orientation-preservingly homeomorphic to that of the given one.
\end{abstract}

\maketitle

\section{Introduction}

This paper is a continuation of the study of cosmetic surgeries on knots in rational homology spheres by the first and second authors in \cite{IchiharaJong2025}. 

Let $M$ be a closed, orientable $3$-manifold and $K$ a knot in $M$.
The operation of forming a new $3$-manifold
$( M - \mathrm{int} N(K) ) \cup_f (D^2 \times S^1)$ from $M$ 
is called \textit{Dehn surgery} on $K$, where $N(K)$ denotes a regular neighborhood of $K$.
The gluing map $f$ is determined by the isotopy class of a loop on $\partial N(K)$ that is identified with the meridian of the attached solid torus. 
Thus, the isotopy class of the loop is called the \textit{surgery slope}. 
For a knot in the $3$-sphere $S^3$, as is well known, such surgery slopes are parametrized by $\mathbb{Q}\cup\{1/0\}$, where $1/0$ corresponds to the meridional slope. 
Dehn surgery on a link is defined in the same way.

In general, Dehn surgeries on a knot along distinct slopes give distinct manifolds, but not always. 
There are some examples of such surgeries on knots in manifolds with boundary and of those yielding manifolds with reversed orientations. 
See \cite{BleilerHodgsonWeeks} for examples. 
In view of this, a pair of Dehn surgeries on a knot is called \textit{purely cosmetic} if the resulting $3$-manifolds are orientation-preservingly homeomorphic.
It is conjectured that no nontrivial knot in a closed, orientable $3$-manifold admits purely cosmetic surgeries along distinct surgery slopes. 
For further background, see \cite{GordonICM, BleilerHodgsonWeeks} and \cite[Problem 1.81A]{Kirby}.

In this paper, we study purely cosmetic surgeries on knots in $3$-manifolds whose first Betti numbers are at most two. 
Our main tool is the rational surgery formula for the Casson--Walker--Lescop invariant $\lambda$, given in \cite{Lescop96}. 
The invariant $\lambda$ was introduced and studied by Lescop \cite{Lescop96} as a generalization of the Casson--Walker invariant \cite{Walker92}. 

In the case of the first Betti number zero, a $3$-manifold $M$ is called an \emph{integral homology sphere} (resp.\ a \emph{rational homology sphere}) if $M$ has the same homology groups as the 3-sphere. That is, 
$H_{0}(M;\mathbb{Z}) = H_{3}(M;\mathbb{Z}) = \mathbb{Z}$
and 
$H_{i}(M;\mathbb{Z}) = \{0\}$ 
for all other $i$ (resp.
$H_{0}(M;\mathbb{Q}) = H_{3}(M;\mathbb{Q}) = \mathbb{Q}$
and
$H_{i}(M;\mathbb{Q}) = \{0\}$
for all other $i$). 

The following result is an extension of \cite[Theorem 2]{IchiharaJong2025}, in which only rational homology spheres obtained by Dehn surgery on a knot in $S^3$ are considered. 
Here, a Dehn surgery is said to be \emph{integral} if the geometric intersection number of the meridian and a representative of the surgery slope is one.

\begin{theorem}\label{thm1}
A null-homologous knot in a rational homology sphere admits at most two pairs of integral purely cosmetic surgeries. 
\end{theorem}

Note that the finiteness of purely cosmetic surgeries on a knot has been proved recently in \cite{Ichihara_2026}. 

In connection with this, the cosmetic surgery problem is closely related to the knot complement problem. 
It is conjectured that if two knots $K_1$ and $K_2$ in a closed, oriented 3-manifold $M$ have complements\footnote{Here, the term ``complement'' is used to mean the ``exterior'' of a knot; that is, it refers to the compact 3-manifold obtained from the ambient manifold by removing an open tubular neighborhood of the knot.} that are homeomorphic via an orientation-preserving homeomorphism, then they are equivalent, i.e., there exists an orientation-preserving homeomorphism of $M$ taking $K_1$ to $K_2$. 
In Section~\ref{sec2}, as an application of our method for computing the Casson--Walker--Lescop invariant, we give some supporting evidence for the conjecture. 
In particular, the following corollaries can be regarded as a partial answer to this conjecture. 
For example, for a null-homologous knot in a 3-manifold obtained by Dehn surgery on a knot in $S^3$, we obtain the following.

\begin{corollary}\label{cor0}
Let $M$ be a $3$-manifold obtained by Dehn surgery on a knot in $S^3$. 
Then, for any null-homologous knot $K$ in $M$, there are at most two knots in $M$ that are inequivalent to $K$ but whose exteriors are orientation-preservingly homeomorphic to the exterior of $K$. 
In particular, for any null-homologous knot in $S^2 \times S^1$ or a lens space, the same holds. 
\end{corollary}

A link in $S^3$ is called \emph{algebraically split} if the linking number of any pair of its components is zero. 

\begin{corollary}\label{cor1}
Let $M$ be a rational homology sphere obtained by Dehn surgery on an algebraically split link in $S^3$ with positive surgery slopes. 
Then, for any null-homologous knot $K$ in $M$, there are at most two knots in $M$ that are inequivalent to $K$ but whose exteriors are orientation-preservingly homeomorphic to the exterior of $K$.
\end{corollary}

Moreover, the next result gives a partial extension and an alternative proof of part of the result obtained in \cite{Gainullin} using Heegaard Floer homology, which states that any null-homologous knot in an L-space is determined by the complement. 

\begin{corollary}\label{cor2}
Let $M$ be a connected sum of lens spaces $L(p_i, q_i)$ with coprime positive integers $p_i$ and $q_i$, or a connected sum of $S^2 \times S^1$ and lens spaces $L(p_i, q_i)$ with coprime integers $p_i$ and $q_i$.  
Then, for any null-homologous knot $K$ in $M$, there are at most two knots in $M$ that are inequivalent to $K$ but whose exteriors are orientation-preservingly homeomorphic to the exterior of $K$. 
\end{corollary}

Note that, for a given knot, the finiteness of inequivalent knots whose exteriors are homeomorphic to that of the given knot is shown in \cite{Martelli}.

\medskip

Next, we consider purely cosmetic surgeries on knots in $3$-manifolds with the first Betti number one or two. 
In what follows, $\hat{a}_k(L)$ denotes the $k$-th coefficient of the Conway polynomial of a link $L$ in $S^3$ used in \cite{Lescop96}. 
These coefficients differ slightly from the usual ones (e.g., those used in \cite{IchiharaJong2025}); we refer the reader to the remark at the end of this section for details.
Also, by \emph{$(r_1,\dots,r_n)$-surgery} on a link $L=K_1 \cup \cdots \cup K_n$ (resp.\ \emph{$r$-surgery} on a knot $K$) in $S^3$, we mean Dehn surgery on $L$ (resp.\ on $K$) along slopes corresponding to the rational numbers $r_1, \dots, r_n$ (resp.\ the slope corresponding to the rational number $r$). 

\begin{theorem}\label{thm3}
Let $K$ and $K_1$ be two knots in $S^3$ such that the link $K \cup K_1$ is algebraically split. 
Let $M$ be the $3$-manifold obtained by $0$-surgery on $K_1$, i.e., surgery along the preferred longitudinal slope. 
Suppose that $\hat{a}_1(K \cup K_1) \ne 0$. 
Then, regarded as a knot in $M$, $K$ admits no purely cosmetic surgeries. 
In particular, $K$ is determined by the complement in $M$. 
\end{theorem}

We remark that the condition $\hat{a}_1(K \cup K_1) \ne 0$ is equivalent to $a_3(K \cup K_1) \ne 0$, where $a_3(K \cup K_1)$ denotes the usual third coefficient of the Conway polynomial of the link $K \cup K_1$ in $S^3$. 
See the remark in the last paragraph of this section. 

\begin{theorem}\label{thm4}
Let $K$, $K_1$, and $K_2$ be knots in $S^3$ such that the link $K \cup K_1 \cup K_2$ is algebraically split. 
Let $M_{p_1/q_1}$ be the 3-manifold obtained by $(p_1/q_1,0)$-surgery on $K_1 \cup K_2$. 
Suppose that $p_1 \hat{a}_1(K \cup K_2) + q_1 \hat{a}_1(K \cup K_1 \cup K_2) \ne 0$. 
Then, regarded as a knot in $M_{p_1/q_1}$, $K$ admits no purely cosmetic surgeries. 
In particular, $K$ is determined by the complement in $M_{p_1/q_1}$. 
\end{theorem}

\begin{theorem}\label{thm5}
Let $K$, $K_1$, and $K_2$ be knots in $S^3$ such that the link $K \cup K_1 \cup K_2$ is algebraically split. 
Let $M_{00}$ be the 3-manifold obtained by $(0,0)$-surgery on $K_1 \cup K_2$. 
Suppose that $\hat{a}_1(K \cup K_1 \cup K_2) \ne 0$. 
Then, regarded as a knot in $M_{00}$, $K$ admits no purely cosmetic surgeries. 
In particular, $K$ is determined by the complement in $M_{00}$. 
\end{theorem}

Theorems~\ref{thm3}--\ref{thm5} may be regarded as generalizations of \cite[Theorem 4]{IchiharaJong2025} to certain $3$-manifolds with the first Betti number one or two. 
For example, in Theorem~\ref{thm3}, when $K_1$ is the unknot, the manifold $M$ is $S^2 \times S^1$. 
In the last section, we give some examples. 
We note that the result above is, in a sense, best possible, because $\hat{a}_1(L')=0$ holds for 
$L' \subset L$ with $\sharp L' \geq 4$ when $L$ is algebraically split \cite[Lemma 1.5]{Hoste}.

\begin{remark}
Here we make a remark about the coefficients of the Conway polynomial, 
specifically the difference between the usual ones and those defined \`{a} la Lescop \cite{Lescop96}. 
The (usual) skein relation of the Conway polynomial $\nabla_L(z)$ of a link $L$ in $S^3$ is given by 
\begin{equation*}
\nabla_{L_+}(z) - \nabla_{L_-}(z) = z \nabla_{L_0}(z).
\end{equation*}
The $k$-th coefficient of $\nabla_L(z)$ is denoted by $a_k(L)$. 
In Lescop's book~\cite{Lescop96}, the notation for the Conway polynomial and its coefficients is slightly different from the above. 
In this paper, we denote by $\hat\nabla_L(z)$ the Conway polynomial used in \cite{Lescop96}, which satisfies the following skein relation:
\begin{equation*}
\hat\nabla_{L_+}(z) - \hat\nabla_{L_-}(z) = -z \hat\nabla_{L_0}(z).
\end{equation*}
If $L$ is a $\mu$-component link, then $\hat\nabla_L(z)$ has the form 
\begin{equation*}
\hat\nabla_{L}(z) = z^{\mu-1} \left(\hat a_0(L) + \hat a_1(L)z^2 + \cdots\right).
\end{equation*}
Note that $\nabla_{L^*}(z) = \hat\nabla_L(z)$, where $L^*$ denotes the mirror image of $L$. 
Also note that $\nabla_{L^*}(z) = (-1)^{\mu -1} \nabla_{L}(z)$.  
Therefore we have 
\begin{equation*}
\hat{a}_k(L) = (-1)^{\mu -1} a_{\mu + 2k -1}(L).
\end{equation*}
In particular, $\hat{a}_k(K) = a_{2k}(K)$ holds for a knot $K$.
\end{remark}

\section{Purely cosmetic surgeries on knots in rational homology spheres}\label{sec2}

In 1985, an integer-valued invariant for oriented integral homology spheres was defined by Casson, which is derived from representations of their fundamental groups into $SU(2)$ (see \cite{AM90}). 
In \cite{Walker92}, Walker extended it to oriented rational homology spheres. 
Then, in \cite{Lescop96}, Lescop gave a formula to calculate the invariant from framed link presentations, and showed that the invariant can be extended to all oriented closed 3-manifolds. 
We call this invariant the Casson--Walker--Lescop invariant, and denote it by $\lambda(M)$ for an oriented closed 3-manifold $M$.

In this section, we first review the surgery formula for the Casson-Walker-Lescop invariant and provide a lemma required for the proof of our results. 
We then proceed to the proofs of Theorem~\ref{thm1} and Corollaries~\ref{cor0}--\ref{cor2}.

\subsection{Surgery formula of $\lambda$}
Let $L = K_1 \cup \cdots \cup K_n$ be an oriented link in $S^3$, and let $M$ be the 3-manifold obtained by $(p_1/q_1 , \dots , p_n/q_n)$-surgery on $L$, where $p_1/q_1 , \dots , p_n/q_n$ are rational numbers. 
The set $( K_i , p_i/q_i )_{1 \le i \le n}$ is called a \emph{surgery presentation} of $M$, and we denote it by $\mathbb{L}$. 
Throughout this paper, we assume that the denominators of the surgery slopes $q_1, \dots, q_n$ are positive unless otherwise noted. 
Then, by \cite[Proposition 1.7.8]{Lescop96}, we have the following surgery formula of the Casson--Walker--Lescop invariant.
\begin{align}\label{eq:LescopFormula}
\begin{aligned}
\lambda(M) =& \ 
(-1)^{b_{-} (\L)} 
\left( \prod_{i=1}^n q_i \right)
\sum_{J \ne \emptyset, J \subset N} 
\det ( E(\mathbb{L}_{N - J} ; J ) )\; \hat{a}_1(L_J) \\
& +(-1)^{b_{-} (\L)} 
\left( \prod_{i=1}^n q_i \right)
\sum_{J \ne \emptyset, J \subset N} 
\frac{\det ( E(\mathbb{L}_{N - J} ) ) (-1)^{\sharp J} \theta(\mathbb{L}_J)}{24} \\
&
+
\left( \prod_{i=1}^n q_i \right)
|\det(E(\mathbb{L}))|
\left(
\frac{\sigma(E(\mathbb{L}))}{8}
+
\sum_{i=1}^{n} 
\frac{s(p_i, q_i) }{2}
\right).
\end{aligned}
\end{align}

Here the terms used in the above formula are defined as follows.
\begin{itemize}
    \item Set $N = \{ 1, \ldots , n \}$.
    \item $l_{ij} = 
    \begin{cases}
    \lk( K_i , K_j) & \text{if } i \ne j, \\
    p_i / q_i  & \text{if } i = j,
    \end{cases}$
    and $E(\mathbb{L}) = ( l_{ij} )_{1 \le i, j \le n}$. 
    Here $\lk(K_i, K_j)$ denotes the linking number of $K_i$ and $K_j$ in $S^3$. 
    Recall that $\lk(K_i, K_j) = \lk(K_j, K_i)$. 
    \item $b_{-} (\L)$ and $\sigma(E(\mathbb{L}))$ denote the number of negative eigenvalues and the signature of $E(\mathbb{L})$, respectively.
    \item
    Set 
    $L_I = \cup_{i \in I} K_i$ for $I \subset N$ 
    and 
    $\mathbb{L}_I =( K_i , p_i/q_i )_{i \in I}$ for $I \subset N$.
    \item Set 
    $ E(\mathbb{L}_{N - J} ; J ) = ( l_{ijJ} )_{i,j \in N - J}$ for $l_{ijJ} = 
    \begin{cases}
        l_{ij} & \text{if } i \ne j, \\
        p_i/q_i + \sum_{k \in J} l_{ki} & \text{if } i = j. 
    \end{cases}$
    \item $s(p,q)$ denotes the Dedekind sum for a pair $(p,q)$ of coprime integers ($q \ne 0$), i.e.,
    \[
    s(p,q) = \sum_{i=1}^{|q|} 
    \left( \left( \frac{i}{q} \right) \right)
    \left( \left( \frac{ip}{q} \right) \right)
    \]
    with
    \[
    \left( \left( x \right) \right)
    = \begin{cases}
        0 & \text{if } x \in \mathbb{Z}, \\
        x - \lfloor x \rfloor - \frac{1}{2} & \text{if } x \not\in \mathbb{Z}.
      \end{cases}
    \]
    By definition, we have 
    \begin{equation}\label{eq:DedekindSum1}
        s(1,q) = \dfrac{(q-1)(q-2)}{12q}
        \quad \text{and} \quad
        s(-1,q) = -\dfrac{(q-1)(q-2)}{12q}. 
    \end{equation}
\end{itemize}
We do not include the exact calculation of $\theta(\L_J)$ in general, since we only need to calculate $\theta(\L_J)$ for the case where there exists a component $K_i$ such that $\lk(K_i,K_j)=0$ for $j \ne i$. 
In this case, we have 
\[
\theta(\mathbb{L}_J) =
\begin{cases}
    \frac{ p_j^2 + q_j^2 + 1 }{ q_j^2 } & \text{if } J = \{ j \}, \\
    0   &   \text{if } \sharp J \ge 2. 
\end{cases}
\]
See \cite{Lescop96} for details.

We remark that the determinant of an empty matrix is set to be one (see \cite{Lescop96}). 
In particular, if $J = N$, it follows that $\det(E(\mathbb{L}_{N - J}))= \det ( E(\mathbb{L}_{N - J} ; J ) ) =1$.  

Now, we let $N=\{ 1,2, \ldots ,n \}$, 
\[
A_N= (l_{ij}) = 
\begin{pmatrix}
 p_1/q_1 & \lk(K_1,K_2) & \cdots  & \lk(K_1, K_n)\\
  \lk(K_2,K_1) & p_2/q_2 & \cdots  & \lk(K_2,K_n) \\
  \vdots & \vdots & \ddots & \vdots \\
  \lk(K_n,K_1) & \lk(K_n,K_2)& \cdots  & p_n/q_n    
\end{pmatrix}, 
\]
$A_{N-J}=(l_{ij})_{i,j \in N-J}$, and 
$B_{N-J}=(b_{ij})_{i,j \in N-J}$ with 
\[
b_{ij}=
\begin{cases}
 l_{ij} & \text{ if } i \ne j, \\
  l_{ii}+\sum_{k \in J} l_{ki} &  \text{ if }  i=j. 
\end{cases}
\]
Then the following lemma, which calculates the Casson-Walker-Lescop invariant, plays a fundamental role in the proofs of our results.

\begin{lemma}\label{lem21}
Let $L=K_0 \cup K_1 \cup \dots \cup K_n$ be an $(n+1)$-component link in $S^3$ satisfying $\lk(K_0,K_j)=0$ for $j \ne 0$. 
Let $M$ be the closed orientable $3$-manifold obtained by $(p_0/q_0,p_1/q_1,\dots,p_{n}/q_{n})$-surgery on $L$ from $S^3$.
Then we have the following. 
\begin{align*}
\lambda(M)
=& \ (-1)^{b_{-} (\L)} \prod_{i=0}^{n} q_i 
\Biggl( \sum_{J \subset N, J \ne \emptyset} \frac{p_0}{q_0} \det B_{N-J} \hat{a}_1(L_{J}) \\
& \qquad  \qquad  \qquad  \qquad  \qquad +\sum_{J' \subset N, J=J' \cup \{ 0\}} \det B_{N-J'} \hat{a}_1(L_{J}) \Biggr)\\
&+ \frac{(-1)^{b_{-} (\L)} \prod_{i=0}^{n}q_i}{24} 
\Biggl( \sum_{J \subset N, J \ne \emptyset}\frac{p_0}{q_0} \det A_{N-J}(-1)^{\sharp J} \theta(\L_{J})\\
&  \qquad  \qquad  \qquad 
-\det A_{N} \left( \frac{p_0^2+q_0^2+1}{q_0^2} \right) - \sum_{i=1}^{n} \frac{p_0}{q_0} \det A_{N-\{i\}} \left( \frac{p_i^2 +q_i^2+1}{q_i^2} \right) \Biggr)\\   
&+ 
\prod_{i=1}^{n} q_i \left| \det A_{N} \right| | p_0 | 
\left( \frac{\sigma(E(\L))}{8}
+\frac{\sum_{i=0}^{n} s(p_i, q_i)}{2} \right).
\end{align*}
\end{lemma}
\begin{proof}
The linking matrix of $\L$ is expressed as 
\[
E(\L) =\left(
\begin{array}{ccccc}
p_0/q_0 &0&0&\cdots &0\\
0 &  p_1/q_1 & \lk(K_1,K_2) & \cdots  & \lk(K_1, K_n)\\
  0&\lk(K_2,K_1) & p_2/q_2 & \cdots  & \lk(K_2,K_n) \\
  \vdots&\vdots & \vdots & \ddots & \vdots \\
  0&\lk(K_n,K_1) & \lk(K_n,K_2)& \cdots  & p_n/q_n
\end{array} 
\right).\]

From this, together with definitions of $A_N$, $A_{N-J}$ and $B_{N - J}$, we have the following. 
\begin{align*}
&\sum_{J \subset N \cup \{ 0 \} , J \ne \emptyset } 
\det ( E(\mathbb{L}_{N - J} ; J ) )\; \hat{a}_1(L_J) \\
& \qquad =
\sum_{J \subset N, J \ne \emptyset}
\frac{p_0}{q_0} 
\det B_{N-J} \hat{a}_1(L_{J}) 
+\sum_{J' \subset N, J=J' \cup \{ 0\}}
\det B_{N-J'} \hat{a}_1(L_{J}) ,
\end{align*}
\begin{align*}
&\prod_{i=0}^n q_i 
\sum_{J \ne \emptyset, J \subset N \cup \{ 0 \} } 
\frac{\det ( E(\mathbb{L}_{N - J} ) ) (-1)^{\sharp J} \theta(\L_J)}{24} \\
&=
\frac{\prod_{i=0}^{n}q_i}{24} 
\Biggl( 
\sum_{J \subset N, J \ne \emptyset}
\frac{p_0}{q_0} \det A_{N-J}(-1)^{\sharp J} \theta(\L_{J})\\
&  \qquad  \qquad  \qquad 
- \det A_{N} \left( \frac{p_0^2+q_0^2+1}{q_0^2} \right)
- \sum_{i=1}^{n} \frac{p_0}{q_0} \det A_{N-\{i\}} \left( \frac{p_i^2 +q_i^2+1}{q_i^2} \right) \Biggr) ,
\end{align*}
and 
\[
\left( \prod_{i=0}^n q_i \right)
|\det(E(\mathbb{L}))|
= \prod_{i=1}^{n} q_i \left| \det A_{N} \right| | p_0 | . 
\]

Consequently, from Lescop's surgery formula~\eqref{eq:LescopFormula}, we have the result. 
\end{proof}

\subsection{Proof of Theorem~\ref{thm1}}
It is well-known that any closed orientable 3-manifold, in particular any rational homology sphere, can be obtained by Dehn surgery on a link in $S^3$~\cite{Lickorish, Wallace}. 
We note that in a $3$-manifold $M$ obtained by Dehn surgery on a link in $S^3$, any null-homologous knot can be isotoped in $M$ so that it bounds a Seifert surface in $M$ disjoint from the dual link; see \cite[Proof of Theorem 2]{IchiharaJong2025}, for example. 
Also note that, for homological reasons, integral purely cosmetic surgeries on a null-homologous knot must be along the slopes $\pm p$ for some positive integer $p$. 
In view of these facts, Theorem~\ref{thm1} follows from the next theorem.

\begin{theorem}\label{thm21}
Let $L = K_0 \cup K_1 \cup \cdots \cup K_n$ be an $(n+1)$-component link in $S^3$ with $n \ge 1$ and $\lk(K_0,K_j)=0$ for $j \ne 0$. 
Suppose that $(p_1/q_1, \dots, p_n/q_n)$-surgery on $K_1 \cup \cdots \cup K_n$ yields a rational homology sphere. 
Then there are at most two positive integers $p$ such that the 3-manifolds obtained by Dehn surgeries on $L$ along slopes $(p/1, p_1/q_1, \dots, p_n/q_n)$ and $(-p/1, p_1/q_1, \ldots , p_n/q_n)$ are orientation-preservingly homeomorphic.
\end{theorem}

\begin{proof}
Set $\lk(K_i, K_j)=l_{ij}$ for $i \ne j$. 
Let $M$ (resp.~$M'$) be the 3-manifold obtained by $( p/1, p_1/q_1, \dots , p_n/q_n) $-surgery (resp.~$( -p/1, p_1/q_1, \dots , p_n/q_n) $-surgery) on $L$ from $S^3$.

Let $\L$ (resp. $\L'$) be the surgery presentation given by $L$ with $( p/1, p_1/q_1, \dots , p_n/q_n)$ (resp. with $( -p/1, p_1/q_1, \dots , p_n/q_n)$). 
Then, we have
\[
(-1)^{b_{-} (\L)} = - (-1)^{b_{-} (\L')} 
\text{ \ and \ } 
\sigma(E(\L)) -2 = \sigma (E(\L') ).
\] 

Together with this, by Lemma~\ref{lem21}, we have the following. 
\begin{align*}
\lambda(M)-\lambda(M')
=& \ 2(-1)^{b_{-} (\L)} \prod_{i=1}^{n}q_i \sum_{J' \subset N, J=J' \cup \{ 0\}}\det B_{N-J'} \hat{a}_1(L_{J})\\
&-\frac{(-1)^{b_{-} (\L)}\prod_{i=1}^{n}q_i}{12} \det A_{N} (p^2+2) +\frac{\prod_{i=1}^{n}q_i | \det A_{N} | }{4}p.
\end{align*}

By the assumption that the 3-manifold obtained by $(p_1/q_1, \dots, p_n/q_n)$-surgery on $K_1 \cup \cdots \cup K_n$ is a rational homology sphere, we see that $\det A_{N} \ne 0$. 
This implies that the right-hand side of the above equality is not identically zero, and hence it is a quadratic polynomial in $p$. 

Therefore, there are at most two positive integers $p$ satisfying that $M$ and $M'$ are orientation-preservingly homeomorphic.
\end{proof}

\subsection{Knot complement problem in a rational homology sphere} 

The cosmetic surgery problem and the knot complement problem are closely related by the following proposition. 
While this fact is part of folklore, we include a proof. 

\begin{proposition}\label{prop:ck}
    Let $K$ be a knot in a closed orientable $3$-manifold $M$. 
    Suppose there are at most $k$ distinct non-meridional slopes $r_1, \dots, r_k$ such that each $r_i$-surgery on $K$ yields a manifold orientation-preservingly homeomorphic to $M$. 
    Then, there exist at most $k$ distinct knots in $M$ that are inequivalent to $K$ but whose exteriors are orientation-preservingly homeomorphic to the exterior of $K$.
    \end{proposition}
\begin{proof}
Suppose there are at most $k$ distinct non-meridional slopes $r_1, \dots, r_k$ such that each $r_i$-surgery on $K$ yields a manifold orientation-preservingly homeomorphic to $M$. 
For each $i = 1, \dots, k$, let $K_i$ be the core of the attached solid torus (the dual knot) of the $r_i$-surgery on $K$. 
Then, for each $i$, the exteriors of $K_i$ and $K$ are orientation-preservingly homeomorphic. 
Also each $K_i$ is a knot in the manifold obtained by $r_i$-surgery, which is orientation-preservingly homeomorphic to $M$, by the assumption. 
Suppose that there exists another knot $K'$ in $M$ such that there exists an orientation preserving homeomorphism $h$ from the exterior of $K'$ to that of $K$ but $K'$ and $K$ are inequivalent knots in $M$. 
Then, the slope $r$ on $\partial N(K)$ determined by $h(\mu_{K'})$ for the meridian $\mu_{K'}$ of $K'$ satisfies that $r$ is non-meridional and $r$-surgery on $K$ yields a manifold orientation-preservingly homeomorphic to $M$. 
However, by the assumption, such slopes for $K$ are only $r_1, \dots, r_k$, and so $r=r_i$ for some $i$. 
It implies that $K'$ is equivalent to one of the knots $K_i$. 
Consequently there exist at most $k$ distinct knots in $M$ that are inequivalent to $K$ but whose exteriors are orientation-preservingly homeomorphic to the exterior of $K$.
\end{proof}

We use Lemma~\ref{lem21} to consider the knot complement problem, and obtain the following. 

\begin{theorem}\label{thm22}
Let $L=K_0 \cup K_1 \cup \cdots \cup K_n$ be an $(n+1)$-component link in $S^3$ with $n \ge 1$ satisfying 
$\lk(K_0,K_j)=0$ for $j \ne 0$. 
For given rational numbers $p_1/q_1, \dots, p_n/q_n$, suppose that at least one of the following two conditions holds:
\[
\sum_{i=1}^{n} \det A_{N-\{ i \}}
\left( \frac{p_i^2 + q_i^2 + 1}{q_i^2} \right) \ne 0, \qquad 
(-1)^{b_{-}(\L')} \det A_{N} - | \det A_{N} | \ne 0 .
\]
Here $\L'$ denotes the surgery presentation given by $L' = L - K_0$ with slopes $( p_1/q_1, \dots , p_n/q_n)$. 
Then, there are at most four positive integers $q$ satisfying that the $3$-manifold obtained by $(1/q, p_1/q_1, \dots, p_n/q_n)$- or $(-1/q, p_1/q_1, \dots, p_n/q_n)$-surgery on $L$ is orientation-preservingly homeomorphic to that obtained by $(p_1/q_1,\dots, p_n/q_n)$-surgery on $L'$. 
\end{theorem}

\begin{proof}
Let $M_{\pm}$ be the 3-manifolds obtained by $( \pm 1/q, p_1/q_1, \dots , p_n/q_n) $-surgeries on $L$ for a positive integer $q$, respectively. 
Let $M$ be the 3-manifold obtained by $(p_1/q_1, \dots , p_n/q_n) $-surgery on $L'$. 
Let $\L_{\pm}$ (resp. $\L'$) be the surgery description given by $L$ with $( \pm 1/q, p_1/q_1, \dots , p_n/q_n)$ (resp. by $L'$ with $( p_1/q_1, \dots , p_n/q_n)$). 
Then, since $\lk(K_0, K_j)=0$ for any $j \ne 0$, we have
\[
(-1)^{b_{-} (\L_{+})} 
= - (-1)^{b_{-} (\L_{-})} 
= (-1)^{b_{-} (\L')} , 
\]
and
\[
\sigma(E(\L_{+})) -1 = \sigma(E(\L_{-})) +1 = \sigma (E(\L')).
\] 

Then, in a similar way as before, we have the following by Lemma~\ref{lem21} and Equation~\eqref{eq:DedekindSum1}. 
\begin{align*}
\lambda(M_{+})-\lambda(M)
&=\left( 
(-1)^{b_{-}(\L')} \prod_{i=1}^{n}q_i 
\sum_{J' \subset N, J=J' \cup \{ 0\}} 
\det B_{N-J'} \hat{a}_1(L_{J}) \right) q \\
& \quad 
-\frac{(-1)^{b_{-}(\L')} \prod_{i=1}^{n}q_i}{24} \det A_{N} \left( \frac{q^2+2}{q} \right) \\
& \quad
-\frac{(-1)^{b_{-}(\L')} \prod_{i=1}^{n}q_i}{24} \sum_{i=1}^{n} \det A_{N-\{ i \}}
\left( \frac{p_i^2 + q_i^2+1}{q_i^2} \right) 
\\   
& \quad 
+ \frac{\prod_{i=1}^{n}q_i}{24} | \det A_{N} | \left( \frac{q^2+2}{q} \right). 
\end{align*}
If $\lambda(M_{+})-\lambda(M)=0$, then we obtain a quadratic equation in $q$ 
\[
c_2 q^2 + c_1 q + c_0 =0
\]
with
\begin{align*}
    c_2 &= 
    (-1)^{b_{-}(\L')} 
\sum_{J' \subset N, J=J' \cup \{ 0\}} 
\det B_{N-J'} \hat{a}_1(L_{J}) \\
& \quad - \frac{1}{24} 
\left( (-1)^{b_{-}(\L')} \det A_{N} - | \det A_{N} | \right) , \\
    c_1 &= -\frac{(-1)^{b_{-}(\L')} 
    }{24} 
\sum_{i=1}^{n} \det A_{N-\{ i \}} \left( \frac{p_i^2 + q_i^2+1}{q_i^2} \right) , \\
    c_0 &= -\frac{1}{12} 
    \left( (-1)^{b_{-}(\L')} \det A_{N} - | \det A_{N} | \right) .
\end{align*}
Thus, if at least one of the two conditions below holds, 
then the quadratic equation above is not identically zero. 
\[
\sum_{i=1}^{n} \det A_{N-\{ i \}} \left( \frac{p_i^2 + q_i^2 + 1}{q_i^2} \right)  \ne 0, 
\qquad 
(-1)^{b_{-}(\L')} \det A_{N} - | \det A_{N} | \ne 0 .
\]
This implies that $\lambda(M_{+})-\lambda(M) = 0$ holds for at most two positive integers $q$. 
This implies that, if at least one of the two conditions above holds, then there are at most two positive integers $q$ satisfying that $M_+$ and $M$ are orientation-preservingly homeomorphic. 

In the same way, we have 
\begin{align*}
\lambda(M_{-})-\lambda(M)
&=
- \left( (-1)^{b_{-}(\L')} 
\prod_{i=1}^{n} q_i 
\sum_{J' \subset N, J=J' \cup \{ 0\}}\det B_{N-J'} \hat{a}_1(L_{J}) \right) q \\
& \quad 
+ \frac{(-1)^{b_{-}(\L')} \prod_{i=1}^{n}q_i}{24} 
\det A_{N} \left( \frac{q^2+2}{q} \right) \\
&-\frac{(-1)^{b_{-}(\L')}\prod_{i=1}^{n}q_i}{24} 
\sum_{i=1}^{n} \det A_{N-\{ i \}}
\left( \frac{p_i^2 + q_i^2+1}{q_i^2} \right) 
\\   
& 
-\frac{\prod_{i=1}^{n}q_i}{24} | \det A_{N} | \left( \frac{q^2+2}{q} \right)  .
\end{align*}
If $\lambda(M_{-})-\lambda(M)=0$, then we obtain a quadratic equation in $q$ 
\[
-c_2 q^2 + c_1 q - c_0 =0. 
\]
Then, in the same way as above, it is shown that there are at most two positive integers $q$ satisfying that $M_-$ and $M$ are orientation-preservingly homeomorphic if at least one of the two conditions above holds. 

Consequently, there are at most four positive integers $q$ satisfying that $M_{+}$ or $M_{-}$ and $M$ are orientation-preservingly homeomorphic if at least one of the following two conditions holds: 
\[
\sum_{i=1}^{n} \det A_{N-\{ i \}} 
\left( \frac{p_i^2 + q_i^2 + 1}{q_i^2} \right) \ne 0, \qquad 
(-1)^{b_{-}(\L')} \det A_{N} - | \det A_{N} | \ne 0 . \qedhere
\]
\end{proof}

Based on the proof of Theorem~\ref{thm22}, we have the following corollary, from which Corollary~\ref{cor1} follows. 

\begin{corollary}\label{cor21}
Let $L = K_0 \cup K_1 \cup \dots \cup K_n$ be an $(n+1)$-component algebraically split link in $S^3$ with $n \ge 1$ and set $L' = K_1 \cup \cdots \cup K_n$. 
Then, for given positive rational numbers $p_1/q_1, \dots, p_n/q_n$, there are at most two positive integers $q$ satisfying that the 3-manifold obtained by $(1/q, p_1/q_1, \dots, p_n/q_n)$-surgery or $(- 1/q, p_1/q_1, \dots, p_n/q_n)$-surgery on $L$ is orientation-preservingly homeomorphic to that obtained by $(p_1/q_1,\dots, p_n/q_n)$-surgery on $L'$. 
\end{corollary}

\begin{proof}
We adopt the same setting as in the proof of Theorem~\ref{thm22}. 
Suppose further that $L$ is an algebraically split link and $p_1/q_1, \dots, p_n/q_n$ are positive rational numbers. 
Then, for $\L'$, namely $L'$ with $(p_1/q_1, \dots, p_n/q_n)$, the linking matrix $E(\L')$ is a diagonal matrix, and we see that 
\[
(-1)^{b_{-} (\L')} \det A_N = | \det A_N | .
\]
This implies that $c_0$ in the proof of Theorem~\ref{thm22} is always 0. 

On the other hand, since $L$ is algebraically split and $p_1/q_1, \dots, p_n/q_n$ are all positive, we have 
\[
\sum_{i=1}^{n} \det A_{N-\{ i \}} \left( \frac{p_i^2 + q_i^2 + 1}{q_i^2} \right)  > 0 , 
\]
that is, $c_1 \ne 0$. 

This implies that $\lambda(M_{+})-\lambda(M) = 0$ holds for at most one positive integer $q$, and so does $\lambda(M_{-})-\lambda(M) = 0$. 
Therefore there are at most two positive integers $q$ satisfying that $M_{\pm}$ and $M$ are orientation-preservingly homeomorphic. 
\end{proof}

Next we consider the case $n=1$ in Theorem~\ref{thm22}. 
In this case, the first condition of Theorem~\ref{thm22} is satisfied, since the determinant of the empty matrix is defined to be $1$. 
As for the second condition, we observe that
\[
\begin{cases}
(-1)^{b_{-}(\L')} = 1,\quad \det A_{N} = \lvert \det A_{N} \rvert & \text{if } \frac{p_1}{q_1} > 0,\\[6pt]
(-1)^{b_{-}(\L')} = 1,\quad \det A_{N} = \lvert \det A_{N} \rvert = 0 & \text{if } \frac{p_1}{q_1} = 0,\\[6pt]
(-1)^{b_{-}(\L')} = -1,\quad \det A_{N} = - \lvert \det A_{N} \rvert & \text{if } \frac{p_1}{q_1} < 0.
\end{cases}
\]
Hence, in all cases, the left-hand side of the condition equals zero. 
Therefore, arguing in the same way as in the proof of Corollary~\ref{cor21}, we obtain Corollary~\ref{cor0}.

In the same way, consider the case where $n \ge 2$, $p_1/q_1 = 0/1$, and $p_i/q_i \neq 0$ for $i = 2,3,\ldots,n$ in Theorem~\ref{thm22}. 
In particular, if $L'$ is algebraically split, then arguing in the same way as in the proof of Corollary~\ref{cor21}, we see that $c_0 = 0$. 
As for $c_1$, since the determinant of the empty matrix is defined to be $1$, it follows that the left-hand side equals $2$ (and hence is not equal to $1$).  
Note that, unlike in Corollary~\ref{cor21}, it is not necessary that the signs of $p_i/q_i$ be the same. 
Using this observation, one obtains, for example, the following statement.

\begin{corollary}
Let $M$ be a connected sum of $S^2 \times S^1$ and lens spaces $L(p_i, q_i)$, where $p_i$ and $q_i$ are coprime integers.  
Then, for any null-homologous knot $K$ in $M$, there are at most two knots in $M$ that are inequivalent to $K$, but whose exteriors are orientation-preservingly homeomorphic to the exterior of $K$. 
\end{corollary}

Together with this and Corollary~\ref{cor21}, since the connected sum of lens spaces is obtained by Dehn surgery on a trivial link, Corollary~\ref{cor2} follows.

\section{Constraints on knots in a 3-manifold with the Betti number 1 or 2}

In this section, we will give proofs of Theorems~\ref{thm3}--\ref{thm5}. 

\subsection{Proof of Theorem~\ref{thm3}}
For a surgery presentation with a two-component link, we give a rational surgery formula from the Lescop formula~\eqref{eq:LescopFormula} as follows. 

Let $K_0 \cup K_1$ be a two-component link with $\lk(K_0,K_1) = l$. 
Let $M$ be the 3-manifold obtained by $(p_0/q_0 , p_1/q_1)$-surgery on $K_0 \cup K_1$, that is, 
\[ \mathbb{L} = \set{ (K_0 , p_0/q_0) , ( K_1 , p_1/q_1) }\] 
is a surgery presentation of $M$. 
Here, $p_0$ and $q_0$ are coprime integers with $p_0 \neq 0$ and $q_0 > 0$. 
Similarly, $p_1$ and $q_1$ are coprime integers with $q_1 > 0$. 
By taking the mirror image if necessary, we may assume that $p_1 \ge 0$. 
In this setting, 
the linking matrix of $\mathbb{L}$ is
\[
E(\mathbb{L}) =
\begin{pmatrix}
    p_0 / q_0 & l \\
    l & p_1 / q_1  
\end{pmatrix},
\]
with 
$\det E(\mathbb{L}) 
= (p_0/q_0) \cdot (p_1/q_1 ) - l^2$ and $\mathrm{tr} E(\mathbb{L}) = p_0/q_0 + p_1/q_1 $. 
It follows that 
\begin{align*}
b_{-} (\L)
&=
\begin{cases}
    0   &  \text{ if } \det E(\mathbb{L}) \ge 0 \text{ and } \mathrm{tr} E(\mathbb{L}) > 0, \\
    1   &  \text{ if } \det E(\mathbb{L}) < 0, \text{ or } \det E(\mathbb{L}) = 0 \text{ and } \mathrm{tr} E(\mathbb{L}) < 0, \\
    2   &  \text{ if } \det E(\mathbb{L}) > 0 \text{ and } \mathrm{tr} E(\mathbb{L}) < 0, 
\end{cases}\\
(-1)^{b_{-} (\L)}
&=
\begin{cases}
    1   &  \text{ if } \det E(\mathbb{L}) > 0, \text{ or } \det E(\mathbb{L}) = 0 \text{ and } \mathrm{tr} E(\mathbb{L}) > 0, \\
    -1   &  \text{ otherwise }, 
\end{cases}
\end{align*}
and 
\[
\sigma(E(\mathbb{L}))
=
\begin{cases}
    2   &  \text{ if } \det E(\mathbb{L}) >0 \text{ and } \mathrm{tr} E(\mathbb{L}) > 0, \\
    1   &   \text{ if } \det E(\mathbb{L}) = 0 \text{ and } \mathrm{tr} E(\mathbb{L}) > 0, \\
    0  &  \text{ if } \det E(\mathbb{L}) < 0, \\
    -1   &  \text{ if } \det E(\mathbb{L}) = 0 \text{ and } \mathrm{tr} E(\mathbb{L}) < 0, \\
    -2    &  \text{ if } \det E(\mathbb{L}) > 0 \text{ and } \mathrm{tr} E(\mathbb{L}) < 0. 
\end{cases} 
\]
Set $N_0 = \{ 0,1 \}$ and consider $J \subset N_0$. 
When $J = \{ 0 \} \subset N_0$, we see that 
\[
\det( E(\mathbb{L}_{N_0 - J}) )
= 
\frac{p_1}{q_1} 
\ , \quad 
\det( E(\mathbb{L}_{N_0 - J} ; J ) ) 
= 
\frac{p_1}{q_1} + l. 
\]
When $J = \{ 1 \} \subset N_0$, we see that 
\[
\det( E(\mathbb{L}_{N_0 - J}) )
= 
\frac{p_0}{q_0}
\ , \quad 
\det( E(\mathbb{L}_{N_0 - J} ; J ) ) 
= 
\frac{p_0}{q_0} + l.
\]
When $J = \{ 0 , 1 \} \subset N_0$, we see that 
\[
\det( E(\mathbb{L}_{N_0 - J}) )
= 
\det( E(\mathbb{L}_{N_0 - J} ; J ) ) 
= 
1.
\]
Also, the following is obtained. See \cite[p.\ 19]{Lescop96} for details. 
\begin{equation}\label{eq:theta}
    \theta(\mathbb{L}_J) =
\begin{cases}
    \frac{ p_0^2 + q_0^2 + 1 }{ q_0^2 }   &   \text{ if } J = \{ 0 \}, \\[6pt]
    \frac{ p_1^2 + q_1^2 + 1 }{ q_1^2 }   &   \text{ if } J = \{ 1 \},\\[6pt]
2 l^3 + 2 l^2 \left( \frac{p_0}{q_0} + \frac{p_1}{q_1} \right) -2l  
    &   \text{ if } J = \{ 0 ,1 \}.  
\end{cases}
\end{equation}
Consequently, we have the following. 
\begin{align}\label{eq3+}
\begin{aligned}
\lambda(M) 
=& \
(-1)^{b_{-} (\L)}
\Biggl(
q_0 q_1 
\left( 
\left( \frac{p_1}{q_1} + l \right) \hat{a}_1(K_0) 
+
\left( \frac{p_0}{q_0} + l \right) \hat{a}_1(K_1) 
+
\hat{a}_1(K_0 \cup K_1) 
\right) \\
& 
- \frac{q_0 q_1}{24} 
\left(
\frac{p_1}{q_1} \cdot \frac{ p_0^2 + q_0^2 + 1 }{ q_0^2 }
+
\frac{p_0}{q_0} \cdot \frac{ p_1^2 + q_1^2 + 1 }{ q_1^2 }
-
\left( 2 l^3 + 2 l^2 \left( \frac{p_0}{q_0} + \frac{p_1}{q_1} \right) - 2 l \right)
\right) \Biggr)\\
&
+ \left| p_0 p_1 - l^2 q_0 q_1 \right| 
\left(
\frac{\sigma(E(\mathbb{L}))}{8}
+
\frac{s(p_0, q_0) }{2}
+
\frac{s(p_1, q_1) }{2}
\right).
\end{aligned}
\end{align}
Actually, this formula with full generality is not necessary in the following. 
We here include it for future use. 
Also see Appendix to compare it with known such formulas, given in \cite{Ito2022, Maruyama12}. 

Then, by Proposition~\ref{prop:ck}, Theorem~\ref{thm3} readily follows from the following two propositions. 

\begin{proposition}
Let $L=K_0 \cup K_1$ be a two-component link in $S^3$ with $\lk(K_0,K_1)=0$. 
Let $q_0$ and $q'_0$ be distinct positive integers coprime to a non-zero integer $p_0$. 
If $\hat{a}_1(L) \ne 0$, then $(p_0/q_0,0)$- and $(p'_0/q'_0,0)$-surgeries on $L$ give distinct manifolds for $p'_0 = \pm p_0$. 
\end{proposition}

\begin{proof}
Under the assumptions that $p_1=0$, $q_1=1$ and $\lk(K_0,K_1)=0$, the RHS of Formula~\eqref{eq3+} is simply expressed as follows. 
\begin{align*}
&\sign(p_0) 
\left(
q_0 
\left( 
\frac{p_0}{q_0} \hat{a}_1(K_1) 
+
\hat{a}_1(K_0 \cup K_1) 
\right) 
- \frac{q_0}{24} 
\left( \frac{2p_0}{q_0} \right) 
\right)\\
&=
\sign(p_0) 
\left(
p_0 \hat{a}_1(K_1) 
+
q_0 \hat{a}_1(K_0 \cup K_1) 
- \frac{p_0}{12} \right), 
\end{align*}
since $(-1)^{b_{-} (\L)} = \sign(p_0)$ in this case, 
where $\sign(p_0) = +1$ (resp.\ $-1$) if $p_0 >0$ (resp.\ $p_0<0$). 

We now consider the 3-manifolds $N$ and $N'$, which are obtained by $( p_0/q_0 , 0 )$- and $(p'_0/q'_0 , 0 )$-surgeries on $L = K_0 \cup K_1$, respectively, where $q_0, q'_0$ are positive integers coprime to a positive integer $p_0$. 
Note that $q_0 \ne q'_0$ holds in this case. 

Then, if $p'_0 = p_0$, $\lambda(N) - \lambda(N') $ is calculated as follows from the above. 
\begin{align*}
&\sign(p_0) 
\left(
\left(
p_0 \hat{a}_1(K_1) 
+
q_0 \hat{a}_1(L) 
- \frac{p_0}{12} \right)
-
\left(
p'_0 \hat{a}_1(K_1) 
+
q'_0 \hat{a}_1(L) 
- \frac{p'_0}{12} \right)
\right) \\
& =
\sign(p_0) 
\left( q_0 - q'_0 \right) \hat{a}_1(L).
\end{align*}
Thus, under the assumption that $\hat{a}_1(L) \ne 0$, $(p_0/q_0,0)$- and $(p'_0/q'_0,0)$-surgeries on $L$ give distinct manifolds. 

Next, we consider the case that $p_0 = -p'_0$. 
Without loss of generality, we may assume that $p_0 >0$. 
In this case, we have
\[
\lambda(N) = p_0 \hat{a}_1(K_1) + q_0 \hat{a}_1(L) - \frac{p_0}{12}, 
\]
and 
\[
\lambda(N') 
= - \left( (-p_0) \hat{a}_1(K_1) + q'_0 \hat{a}_1(L) - \frac{-p_0}{12} \right)
= p_0 \hat{a}_1(K_1) - q'_0 \hat{a}_1(L) - \frac{p_0}{12}. 
\]
It follows that 
\[
\lambda(N) - \lambda(N') =
( q_0 + q'_0 ) \hat{a}_1(L) . 
\]
Thus, under the assumption that $\hat{a}_1(L) \ne 0$, 
$(p_0/q_0,0)$- and $(p'_0/q'_0,0)$-surgeries on $L$ give distinct manifolds. 
\end{proof}

\begin{proposition}
Let $L=K_0 \cup K_1$ be a two-component link in $S^3$ with $\lk(K_0,K_1)=0$. 
Let $q_0$ be a positive integer coprime to a non-zero integer $p_0$. 
If $\hat{a}_1(L) \ne 0$, then $(\pm 1/q_0,0)$-surgery on $L$ and $0$-surgery on $K_1$ give distinct manifolds. 
\end{proposition}

\begin{proof}
We consider the 3-manifolds $N$ and $N'$, which are obtained by $( p_0/q_0 , 0 )$-surgery on $L = K_0 \cup K_1$ for $p_0 = \pm 1$ and a positive integer $q_0$ and by $0$-surgery on the knot $K_1$, respectively. 
From Formula~\eqref{eq3+}, as before, we have the following. 
\begin{align*}
\lambda(N) 
&=\sign(p_0) 
\left(
q_0 
\left( 
\frac{p_0}{q_0} \hat{a}_1(K_1) 
+
\hat{a}_1(K_0 \cup K_1) 
\right) 
- \frac{q_0}{24} 
\left( \frac{2p_0}{q_0} \right) 
\right)\\
& =
\sign(p_0) 
\left(
p_0 \hat{a}_1(K_1) 
+
q_0 \hat{a}_1(K_0 \cup K_1) 
- \frac{p_0}{12} \right) \\
& =
\hat{a}_1(K_1) 
+
\sign(p_0) q_0 \hat{a}_1(K_0 \cup K_1) 
- \frac{1}{12}. 
\end{align*}
We remark that $\sign(p_0) = p_0 = \pm 1$.

On the other hand, setting $p_1=0$ and $q_1=1$, we have the following from the Lescop formula~\eqref{eq:LescopFormula}. (Also see Appendix.) 
\[
\lambda(N') = \hat{a}_1(K) - \frac{1}{12} . 
\]
Thus, the following is obtained. 
\[
\lambda(N) - \lambda(N') 
= \sign(p_0) q_0 \hat{a}_1(L) 
= \pm q_0 \hat{a}_1(L).
\]
It implies that, under the assumption that $\hat{a}_1(L) \ne 0$, $(\pm 1/q_0,0)$-surgery on $L$ and $0$-surgery on $K_1$ give distinct manifolds. 
\end{proof}

\subsection{Proof of Theorem~\ref{thm4}}
Similarly as above, to prove Theorem~\ref{thm4}, for a surgery presentation with an algebraically split three-component link, we give a rational surgery formula from the Lescop formula~\eqref{eq:LescopFormula} as follows. 

Let $M$ be the 3-manifold obtained by $(p_0/q_0, p_1/q_1, p_2/q_2)$-surgery on $L=K_0 \cup K_1 \cup K_2$ in $S^3$, where $L$ is an algebraically split link. 
Setting $N_0 = \{ 0, 1 , 2 \}$ and $\mathbb{L}=\set{(K_i, p_i/q_i)}_{i \in N_0}$, we apply Formula~\eqref{eq:LescopFormula}. 
The linking matrix $E(\mathbb{L})$ is 
    \[
    E(\mathbb{L}) = 
    \begin{pmatrix}
        p_0/q_0 & 0 & 0 \\
        0 & p_1/q_1 & 0\\
        0 & 0 & p_2/q_2
    \end{pmatrix}, 
    \]
and so, 
$\det (E(\mathbb{L}) ) = \dfrac{p_0 p_1 p_2}{q_0 q_1 q_2}$, and 
\begin{align*}
\det( E(\mathbb{L}_{N - J}) )= \det( E(\mathbb{L}_{N - J} ; J ) )
    &=
    \begin{cases}
    \dfrac{p_1 p_2}{q_1 q_2} & \text{ if } J = \{ 0 \}, \\[10pt] 
    \dfrac{p_0 p_2}{q_0 q_2} & \text{ if } J = \{ 1 \}, \\[10pt]  
    \dfrac{p_0 p_1}{q_0 q_1} & \text{ if } J = \{ 2 \}, \\[10pt]  
    \dfrac{p_2}{q_2} & \text{ if } J = \{ 0, 1\}, \\[10pt]  
    \dfrac{p_1}{q_1} & \text{ if } J = \{ 0, 2\}, \\[10pt]  
    \dfrac{p_0}{q_0} & \text{ if } J = \{ 1, 2\}.
    \end{cases}
\end{align*}
Since $\lk(K_i,K_j)=0$ for $i \ne j$ ($ 0 \le i,j \le 2$), 
\[
\theta(\mathbb{L}_J) =
\begin{cases}
    \frac{ p_j^2 + q_j^2 + 1 }{ q_j^2 } & \text{ if $J = \{ j \}$ for $j = 0,1,2$} ,\\
    0   &   \text{ if $\sharp J \ge 2$} .
\end{cases}
\]
From the above, we obtain the following. 
\begin{align*}
\lambda(M)
& =(-1)^{b_{-} (\L)} \Bigl( q_0 p_1 p_2 \hat{a}_1(K_0) 
+ p_0 q_1 p_2 \hat{a}_1(K_1) + p_0 p_1 q_2 \hat{a}_1(K_2) \\
& \quad + q_0 q_1 p_2 \hat{a}_1(K_0 \cup K_1) + q_0 p_1 q_2 \hat{a}_1(K_0  \cup  K_2) +p_0 q_1 q_2 \hat{a}_1(K_1 \cup K_2)\\
& \quad +q_0 q_1 q_2 \hat{a}_1(K_0 \cup K_1 \cup K_2) \Bigr)\\
& \quad +\frac{(-1)^{b_{-} (\L)}}{24} \Biggl( - \frac{p_1 p_2}{q_0} \left( p_0^2 + q_0^2 + 1 \right) -\frac{p_0 p_2}{q_1} \left( p_1^2 + q_1^2 + 1 \right)\\
&\quad - \frac{p_0 p_1}{q_2} \left( p_2^2 + q_2^2 + 1 \right) \Biggr) \\
& \quad +|p_0 p_1 p_2| \left( \frac{\sigma( E(\L))}{8} +\frac{s(p_0, q_0)}{2}  + \frac{s(p_1,q_1)}{2} + \frac{s(p_2,q_2)}{2} \right) .
\end{align*}

In the following, we consider the case that $p_2=0$ and $q_2=1$ in the above formula. 
Let $M_0$ be the 3-manifold obtained by $(p_0/q_0, p_1/q_1, 0/1)$-surgery on $L$. 
Then, its Casson--Walker--Lescop invariant is calculated as follows. 
\begin{align}\label{fml32}
\begin{aligned}
\lambda(M_0)
&= 
(-1)^{b_{-} (\L)} 
\Bigl( p_0 p_1 \hat{a}_1(K_2) 
+ q_0 p_1 \hat{a}_1(K_0 \cup K_2) \\
& \qquad \qquad +p_0 q_1 \hat{a}_1(K_1 \cup K_2)
+ q_0 q_1 \hat{a}_1(K_0 \cup K_1 \cup K_2) 
- \frac{p_0 p_1}{12} \Bigr). 
\end{aligned}
\end{align}

Now Theorem~\ref{thm4} follows from the next two propositions. 

\begin{proposition}
Let $L = K_0 \cup K_1 \cup K_2$ be an algebraically split, three-component link in $S^3$.
Let $q_0$ and $q'_0$ be distinct positive integers coprime to a non-zero integer $p_0$ and set $p'_0 = \pm p_0$. 
For a non-zero rational number $p_1/q_1$, if $p_1 \hat{a}_1(K_0 \cup K_2) + q_1 \hat{a}_1(K_0 \cup K_1 \cup K_2) \ne 0$, then $(p_0/q_0, p_1/q_1, 0/1)$- and $( p'_0/q'_0, p_1/q_1, 0/1)$-surgeries on $L$ give distinct manifolds.     
\end{proposition}

\begin{proof}
Let $N$ and $N'$ be the 3-manifolds obtained by $(p_0/q_0, p_1/q_1, 0/1)$- and $(p'_0/q'_0, p_1/q_1, 0/1)$-surgeries on $L$, respectively, assuming that $p'_0 = \pm p_0$ and $p_1/q_1 \ne 0$.  
Taking the mirror image if necessary, we may assume that $p_1>0$. 

Suppose first that $p'_0 = p_0$. 
Since $(-1)^{b_{-} (\L)} = \sign(p_0)$ in this case, we have the following. 
\begin{align*}
\lambda(N) - \lambda(N')
=& \ 
\sign(p_0)
\Biggl(
\Bigl( p_0 p_1 \hat{a}_1(K_2) 
+ q_0 p_1 \hat{a}_1(K_0 \cup K_2) \\
& \ +p_0 q_1 \hat{a}_1(K_1 \cup K_2)
+ q_0 q_1 \hat{a}_1(K_0 \cup K_1 \cup K_2) 
- \frac{p_0 p_1}{12} \Bigr) \\
& \ -
\Bigl( p_0 p_1 \hat{a}_1(K_2) 
+ q'_0 p_1 \hat{a}_1(K_0 \cup K_2) + p_0 q_1 \hat{a}_1(K_1 \cup K_2)\\
& \ + q'_0 q_1 \hat{a}_1(K_0 \cup K_1 \cup K_2) 
- \frac{p_0 p_1}{12} \Bigr)
\Biggr)\\
=& 
\sign(p_0)
( q_0 - q'_0) \left(
p_1 \hat{a}_1(K_0 \cup K_2) 
+ q_1 \hat{a}_1(K_0 \cup K_1 \cup K_2) 
\right). 
\end{align*}

Next suppose that $p'_0 = - p_0$. 
Without loss of generality, we may assume that $p_0>0$. 
Then, we have the following. 
\begin{align*}
\lambda(N) - \lambda(N')
=& \  
\Bigl( p_0 p_1 \hat{a}_1(K_2) 
+ q_0 p_1 \hat{a}_1(K_0 \cup K_2) \\
& \ +p_0 q_1 \hat{a}_1(K_1 \cup K_2)
+ q_0 q_1 \hat{a}_1(K_0 \cup K_1 \cup K_2) 
- \frac{p_0 p_1}{12} \Bigr) \\
& \ -
(-1)
\Bigl( (-p_0) p_1 \hat{a}_1(K_2) 
+ q'_0 p_1 \hat{a}_1(K_0 \cup K_2) \\
& \ + (-p_0) q_1 \hat{a}_1(K_1 \cup K_2)
+ q'_0 q_1 \hat{a}_1(K_0 \cup K_1 \cup K_2) 
- \frac{(-p_0) p_1}{12} \Bigr)\\
=& \ 
( q_0 + q'_0) \left(
p_1 \hat{a}_1(K_0 \cup K_2) 
+ q_1 \hat{a}_1(K_0 \cup K_1 \cup K_2) 
\right).  
\end{align*}

This implies that, in both cases, $(p_0/q_0,p_1/q_1, 0/1)$- and $(p_0'/q'_0, p_1/q_1, 0/1)$-surgeries on $L$ give distinct manifolds if $p_1 \hat{a}_1(K_0 \cup K_2) + q_1 \hat{a}_1(K_0 \cup K_1 \cup K_2) \ne 0$. 
\end{proof}

\begin{proposition}
Let $L = K_0 \cup K_1 \cup K_2$ be an algebraically split, three-component link in $S^3$.
For a non-zero rational number $p_1/q_1$ and for a positive integer $q_0$, if $p_1 \hat{a}_1(K_0 \cup K_2) + q_1 \hat{a}_1(K_0 \cup K_1 \cup K_2) \ne 0$, then $(\pm 1/q_0, p_1/q_1, 0/1)$-surgery on $L$ and $( p_1/q_1, 0/1)$-surgery on $K_1 \cup K_2$ give distinct manifolds.     
\end{proposition}

\begin{proof}
Let $N$ and $N'$ be the 3-manifolds obtained by $(p_0/q_0, p_1/q_1, 0/1)$-surgery on $L$ and $( p_1/q_1, 0/1)$-surgery on $K_1 \cup K_2$, respectively, assuming that $p_0 = \pm 1$ and $p_1/q_1 \ne 0$.  
Taking the mirror image if necessary, we may assume that $p_1>0$. 

For $N$, since $(-1)^{b_{-} (\L)} = \mathrm{sign}(p_0)=p_0$, we have the following from Formula~\eqref{fml32}.
\begin{align*}
\lambda(N) 
=& \ 
p_0
\Bigl( p_0 p_1 \hat{a}_1(K_2) 
+ q_0 p_1 \hat{a}_1(K_0 \cup K_2) \\
& \ +p_0 q_1 \hat{a}_1(K_1 \cup K_2)
+ q_0 q_1 \hat{a}_1(K_0 \cup K_1 \cup K_2) 
- \frac{p_0 p_1}{12} \Bigr) \\
=& \ 
p_1 \hat{a}_1(K_2) 
+ p_0 q_0 p_1 \hat{a}_1(K_0 \cup K_2) \\
& \ + q_1 \hat{a}_1(K_1 \cup K_2)
+ p_0 q_0 q_1 \hat{a}_1(K_0 \cup K_1 \cup K_2) 
- \frac{p_1}{12}.  
\end{align*}

For $N'$, we have the following from Formula~\eqref{eq3+}.
\begin{align*}
\lambda(N') 
= & \ 
q_1 q_2 
\left( 
\frac{p_2}{q_2} \hat{a}_1(K_1) 
+
\frac{p_1}{q_1} \hat{a}_1(K_2) 
+
\hat{a}_1(K_1 \cup K_2) 
\right) \\
& 
- \frac{q_1 q_2}{24} 
\left(
\frac{p_2}{q_2} \cdot \frac{ p_1^2 + q_1^2 + 1 }{ q_1^2 }
+
\frac{p_1}{q_1} \cdot \frac{ p_2^2 + q_2^2 + 1 }{ q_2^2 }
\right) \\
&
+ \left| p_1 p_2 \right| 
\left(
\frac{\sigma(E(\mathbb{L}))}{8}
+
\frac{s(p_1, q_1) }{2}
+
\frac{s(p_2, q_2) }{2}
\right) \\
=& \  
q_1 
\left( 
\left( \frac{p_1}{q_1} \right) \hat{a}_1(K_2) 
+
\hat{a}_1(K_1 \cup K_2) 
\right)  
- \frac{q_1}{24} 
\left(
2 \frac{p_1}{q_1} 
\right) \\
=& \ 
p_1 \hat{a}_1(K_2) 
+
q_1 \hat{a}_1(K_1 \cup K_2) 
- 
\frac{p_1}{12} .
\end{align*}

Thus, we have
\[
\lambda(N) - \lambda(N')  = 
p_0 q_0 \left( p_1 \hat{a}_1(K_0 \cup K_2) 
+ q_1 \hat{a}_1 (K_0 \cup K_1 \cup K_2) \right). 
\]

Consequently, $(\pm 1/q_0, p_1/q_1, 0/1)$-surgery on $L$ and $( p_1/q_1, 0/1)$-surgery on $K_1 \cup K_2$ give distinct manifolds if $p_1 \hat{a}_1(K_0 \cup K_2) + q_1 \hat{a}_1(K_0 \cup K_1 \cup K_2) \ne 0$. 
\end{proof}

\subsection{Proof of Theorem~\ref{thm5}}

We next consider the case that $p_1=0$ and $q_1 = 1$ in Formula~\eqref{fml32}. 
Let $M_{00}$ be the 3-manifold obtained by $(p_0/q_0, 0/1, 0/1)$-surgery on $L$. 
Then, its Casson--Walker--Lescop invariant is calculated as follows. 
\[
\lambda(M_{00})
=(-1)^{b_{-} (\L)} \left( p_0 \hat{a}_1(K_1 \cup K_2)  + q_0  \hat{a}_1(K_0 \cup K_1 \cup K_2 ) \right)
\]
Then Theorem~\ref{thm5} follows from the following two propositions. 

\begin{proposition}
Let $L = K_0 \cup K_1 \cup K_2$ be an algebraically split, three-component link in $S^3$.
Let $q_0$ and $q'_0$ be distinct positive integers coprime to a non-zero integer $p_0$ and set $p'_0 = \pm p_0$. 
If $\hat{a}_1(K_0 \cup K_1 \cup K_2) \ne 0$, then $(p_0/q_0, 0, 0)$- and $( p'_0/q'_0, 0, 0)$-surgeries on $L$ give distinct manifolds.
\end{proposition}

\begin{proof}
Let $N$ and $N'$ be the 3-manifolds obtained by $(p_0/q_0,, 0, 0)$- and $(p'_0/q'_0, 0, 0)$-surgeries on $L$, respectively, assuming that $p'_0 = \pm p_0$. 
Then, $(-1)^{b_{-} (\L)} = \sign(p_0)$ holds in this case. 

If $p'_0 = p_0$, then we have
\begin{align*}
\lambda(N) - \lambda(N')
=& \ 
\sign(p_0)
\Bigl( \left( p_0 \hat{a}_1(K_1 \cup K_2)  + q_0  \hat{a}_1(K_0 \cup K_1 \cup K_2) \right)\\
& - \ \left( p_0 \hat{a}_1(K_1 \cup K_2)  + q'_0  \hat{a}_1(K_0 \cup K_1 \cup K_2) \right) \Bigr)\\
=& \ \sign(p_0) ( q_0 - q'_0) \hat{a}_1(K_0 \cup K_1 \cup K_2)  
\end{align*}

If $p'_0 = - p_0$, we may assume without loss of generality that $p_0>0$. 
Then, we have
\begin{align*}
\lambda(N) - \lambda(N')
=& \ 
\Bigl( \left( p_0 \hat{a}_1(K_1 \cup K_2)  + q_0  \hat{a}_1(K_0 \cup K_1 \cup K_2) \right)\\
&- \ 
(-1)\left( (-p_0) \hat{a}_1(K_1 \cup K_2)  + q'_0  \hat{a}_1(K_0 \cup K_1 \cup K_2) \right) \Bigr)\\
=& \ 
( q_0 + q'_0) \hat{a}_1(K_0 \cup K_1 \cup K_2) 
\end{align*}

Therefore, in both cases, $(p_0/q_0, 0, 0)$- and $( p'_0/q'_0, 0, 0)$-surgeries on $L$ give distinct manifolds if $\hat{a}_1(K_0 \cup K_1 \cup K_2) \ne 0$. 
\end{proof}

\begin{proposition}
Let $L = K_0 \cup K_1 \cup K_2$ be an algebraically split, three-component link in $S^3$.
Let $q_0$ be a positive integer and set $p_0 = \pm 1$. 
If $\hat{a}_1(K_0 \cup K_1 \cup K_2) \ne 0$, then $(p_0/q_0, 0, 0)$-surgery on $L$ and $(0, 0)$-surgery on $K_1 \cup K_2$ give distinct manifolds.
\end{proposition}

\begin{proof}
Let $N$, $N'$ be the 3-manifolds obtained by $(p_0/q_0, 0, 0)$-surgery on $L$ and $( 0, 0)$-surgery on $K_1 \cup K_2$, respectively, assuming that $p_0 = \pm 1$. 

For $N$, since $(-1)^{b_{-} (\L)} = \mathrm{sign}(p_0)=p_0$, we have the following from Formula~\eqref{fml32}.
\begin{align*}
\lambda(N) 
&= 
p_0
\Bigl( p_0 \hat{a}_1(K_1 \cup K_2)
+ q_0 \hat{a}_1(K_0 \cup K_1 \cup K_2) \Bigr) \\
&= 
\hat{a}_1(K_1 \cup K_2)
+ p_0 q_0 \hat{a}_1(K_0 \cup K_1 \cup K_2) .
\end{align*}

For $N'$, we have the following from Formula~\eqref{eq3+}.
\[
\lambda(N') 
= \hat{a}_1(K_1 \cup K_2). 
\]

Thus, we have
\[
\lambda(N) - \lambda(N')  = 
p_0 q_0 \hat{a}_1(K_0 \cup K_1 \cup K_2). 
\]

Consequently, $(\pm 1/q_0, 0, 0)$-surgery on $L$ and $( 0, 0)$-surgery on $K_1 \cup K_2$ give distinct manifolds if $\hat{a}_1(K_0 \cup K_1 \cup K_2) \ne 0$. 
\end{proof}

\section{Examples}

We exhibit some examples concerning Theorems~\ref{thm3}--\ref{thm5}. 

\begin{example}
Let $W$ be the Whitehead link in $S^3$ as depicted in Figure~\ref{fig:WhiteheadLink}. 
Its Conway polynomial is calculated as $\nabla_{W} (z) = z^3$, and so $\hat{a}_1(W) = -a_3(W) = -1$. 
Then, by Theorem~\ref{thm3}, the knot in $S^2 \times S^1$, which is coming from one component of $W$, has no purely cosmetic surgery, since $S^2 \times S^1$ is obtained by $0$-surgery on another component of $W$. 
In addition, let $L_m$ be the link as depicted in Figure~\ref{fig:WhiteheadLink}, where $m$ is a non-zero integer. 
Note that $L_1 = W$, and we have $\nabla_{L_m} (z) = m z^3$. 
Then, by Theorem~\ref{thm3}, the knot in $S^2 \times S^1$ as depicted in Figure~\ref{fig:Lm} has no purely cosmetic surgery. 
\end{example}

\begin{figure}[htb!]
    \centering
    \begin{overpic}[width=0.5\linewidth]{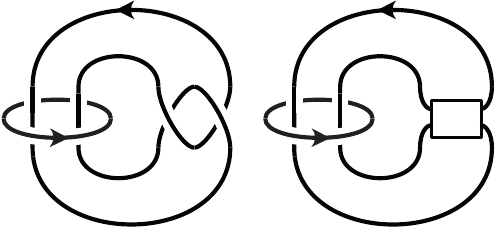}
    \put(89.7,20.5){$m$}    
    \end{overpic}    
    \caption{Whitehead link $W$ and the link $L_m$. The box labeled $m$ indicates $m$ times full-twists.}
    \label{fig:WhiteheadLink}
\end{figure}

\begin{figure}[htb!]
    \centering
    \begin{overpic}[width=0.12\linewidth]{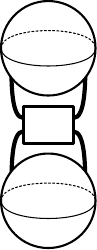}
    \put(15.5,47.8){$m$}    
    \end{overpic}    
    \caption{Identifying two $2$-spheres, we have a knot in $S^2 \times S^1$, which admits no purely cosmetic surgeries.}
    \label{fig:Lm}
\end{figure}

\begin{example}
Next we consider the Borromean rings $B$. 
Its Conway polynomial is calculated as $\nabla_B(z) = z^4$, and so $\hat{a}_1(B) = a_4(B) = 1$. 
By Theorem~\ref{thm4}, the knot coming from one component of $B$ in the connected sum of $S^2 \times S^1$ and any lens space has no purely cosmetic surgery. 
Also, by Theorem~\ref{thm5}, the knot coming from one component of $B$ in the connected sum of two $S^2 \times S^1$'s has no purely cosmetic surgery. 
\end{example}

\appendix 
\section{Confirming several formulae}

We first recover the surgery formula of the Casson invariant given by Boyer and Lines \cite{BoyerLines} from Lescop's formula~\eqref{eq:LescopFormula}. 

Let $M$ be the 3-manifold obtained by $p/q$-surgery on a knot $K$ in $S^3$. 
Suppose that $p,q>0$ for simplicity. 
Then, from Lescop's formula~\eqref{eq:LescopFormula}, we have the following. 
\begin{align*}
\lambda(M) 
&=
 q \hat{a}_1(K) - q \left( \frac{p^2 + q^2 +1}{24 q^2} \right) + q \left| \frac{p}{q} \right| \left( \frac{1}{8} + \frac{s(p,q)}{2} \right) \\
&=
q \hat{a}_1(K) - \left( \frac{p^2 + q^2 +1}{24 q} \right) + p \left( \frac{1}{8} + \frac{s(p,q)}{2} \right) .  
\end{align*}

The right-hand side above is transformed as follows. 
\begin{align*}
& q \hat{a}_1(K) - \left( \frac{p^2 + q^2 +1}{24 q} \right) + p \left( \frac{1}{8} + \frac{s(p,q)}{2} \right)\\
&=
q a_2(K) -\frac{p^2+q^2+1}{24q}
+\frac{p}{8}+\frac{p}{2} \cdot \frac{p^2+q^2+1-3pq}{12pq}-\frac{p}{2}s(q,p) \\
&=
q a_2(K)-\frac{p^2+q^2+1}{24q}+\frac{p}{8}+\frac{p^2+q^2+1-3pq}{24q}-\frac{p}{2}s(q,p)\\
&=
qa_2(K)+\frac{-p^2-q^2-1+3pq+p^2+q^2+1-3pq}{24q}-\frac{p}{2}s(q,p)\\
&=q a_2(K)-\frac{p}{2}s(q,p)    
\end{align*}
Here we used the following reciprocity law of the Dedekind sum. 
\[
s(p,q)+s(q,p)=\frac{p^2+q^2+1-3pq}{12pq}. 
\]

Thus, we can recover the following formula given in \cite{BoyerLines}. 
\[
\lambda (M) = q a_2(K)-\frac{p}{2}s(q,p).    
\]

Next we confirm that the rational surgery formula of the Casson--Walker invariant $\lambda_w$ for 2-component links in $S^3$ given by Ito~\cite[Theorem 1.5]{Ito2022} is derived from Equation~\eqref{eq:LescopFormula}. 

Let $L=K_1 \cup K_2$ be a 2-component link in $S^3$ with $l = \lk(K_1, K_2)$. 
Let $M$ be a 3-manifold with a surgery presentation $\L = \set{(K_1, p_1/q_1), (K_2, p_2/q_2)}$. 
Suppose that $q_1, q_2 > 0$. 
As in \cite[p.13, T5.0]{Lescop96}, the Casson--Walker invariant $\lambda_w$ and the Casson--Walker--Lescop invariant $\lambda$ of $M$ are related by 
\begin{align}
    \lambda(M) &= \dfrac{|H_1(M)|}{2} \lambda_w(M) \notag \\
    &=(-1)^{b_-(\L)}q_1 q_2\left( \dfrac{p_1p_2}{q_1q_2} -l^2\right) \dfrac{\lambda_w(M)}{2}. \label{eq:CW-CSL}
\end{align}
The following formula of $\lambda_w$ is due to Ito~\cite{Ito2022}. 
\begin{proposition}[{\cite[Theorem 1.5]{Ito2022}}]\label{prop:Ito} 
If a rational homology 3-sphere $M$ is represented by $\L = \set{(K_1,p_1 / q_1), (K_2,p_2 / q_2)}$, then its Casson--Walker invariant $\lambda_w(M)$ is given by 
\begin{align*}
\begin{aligned}
&\left(\dfrac{p_1p_2}{q_1q_2} -l^2 \right) \left(\frac{\lambda_w(M)}{2}-\frac{\sigma(E(\L))}{8}\right) \\
&= a_2(K_1)\frac{p_2}{q_2}-\frac{p_2}{24q_2} -\frac{p_2}{24q_2q_1^2} +\frac{p_2 l^2}{24q_2}
+a_2(K_2)\frac{p_1}{q_1} -\frac{p_1}{24q_1}-\frac{p_1}{24q_1q_2^2}+\frac{p_1 l^2}{24q_1}\\
&\quad -a_3(L)+(a_2(K_1)+a_2(K_2)) l +\frac{1}{12}(l^3-l) \\ 
&\quad + \left(\dfrac{p_1p_2}{q_1q_2} -l^2 \right)\left( \dfrac{s(p_1,q_1)}{2} -\frac{p_1}{24q_1} + \dfrac{s(p_2,q_2)}{2}-\frac{p_2}{24q_2} \right).
\end{aligned}
\end{align*}    
\end{proposition}
Proposition~\ref{prop:Ito} yields that 
\begin{align}\label{eq:Ito}
\begin{aligned}
&(-1)^{b_-(\L)}q_1 q_2\left( \dfrac{p_1p_2}{q_1q_2} -l^2\right) \dfrac{\lambda_w(M)}{2}\\
&= (-1)^{b_-(\L)}q_1 q_2 \Biggl(a_2(K_1)\frac{p_2}{q_2}-\frac{p_2}{24q_2} -\frac{p_2}{24q_2q_1^2} +\frac{p_2 l^2}{24q_2} \\
& \qquad \qquad \qquad \qquad +a_2(K_2)\frac{p_1}{q_1} -\frac{p_1}{24q_1}-\frac{p_1}{24q_1q_2^2}+\frac{p_1 l^2}{24q_1} \\
&\qquad \qquad \qquad \qquad -a_3(L)+(a_2(K_1)+a_2(K_2)) l +\frac{1}{12}(l^3-l) \\ 
&\qquad + \left(\dfrac{p_1p_2}{q_1q_2} -l^2 \right)\left( \frac{\sigma(E(\L))}{8} + \dfrac{s(p_1,q_1)}{2}-\frac{p_1}{24q_1} + \dfrac{s(p_2,q_2)}{2}-\frac{p_2}{24q_2} \right)\Biggr). 
\end{aligned}
\end{align}

On the other hand, applying Formula~\eqref{eq:LescopFormula} for $\L = \set{(K_1,p_1 / q_1), (K_2,p_2 / q_2)}$, we have 
\begin{align}\label{eq:LescopFor2comp}
\begin{aligned}
    &\lambda(M) \\
    &=(-1)^{b_-(\L)} q_1q_2 \left( 
    \det E(\L_{\{2\}};\{1\}) \hat{a}_1(L_{\{1\}}) 
    +\det E(\L_{\{1\}};{\{2\}}) \hat{a}_1(L_{\{2\}}) 
    +\hat{a}_1(K_1 \cup K_2) \right) \\ 
    &\quad +\frac{(-1)^{b_-(\L)} q_1 q_2}{24}\left(-\frac{p_2}{q_2}\left(\frac{p_1^2+q_1^2+1}{q_1^2}\right)-\frac{p_1}{q_1}\left(\frac{p_2^2+q_2^2+1}{q_2^2}\right) +\theta(\L_{\{1,2\}}) \right) \\ 
    &\quad +|p_1p_2-q_1q_2l^2|\left(\frac{\sigma(E(\L))}{8} + \frac{s(p_1,q_1)}{2} + \frac{s(p_2,q_2)}{2} \right)  \\ 
    &= (-1)^{b_-(\L)} q_1q_2 \left( \left( \frac{p_2}{q_2} + l \right) \hat{a}_1(K_1) + \left( \frac{p_1}{q_1} +l \right) \hat{a}_1(K_{2}) +\hat{a}_1(K_1 \cup K_2) \right) \\ 
    &\quad +\frac{(-1)^{b_-(\L)} q_1 q_2}{24}
    \Biggl( -\frac{p_2}{q_2}\left(\frac{p_1^2+q_1^2+1}{q_1^2} \right)-\frac{p_1}{q_1}\left(\frac{p_2^2+q_2^2+1}{q_2^2}\right) \\
    & \quad +2 l^2 \left( \frac{p_1}{q_1} + \frac{p_2}{q_2} \right) + 2 l^3 -2l \Biggr) 
    +\left| p_1 p_2 - q_1 q_2 l^2 \right| \left(\frac{\sigma(E(\L))}{8} + \frac{s(p_1,q_1)}{2} + \frac{s(p_2,q_2)}{2} \right) \\ 
    &= (-1)^{b_-(\L)} q_1 q_2\Biggl(\frac{p_2}{q_2}\left(a_2(K_1) -\frac{p_1^2+q_1^2+1}{24q_1^2} \right)+\frac{p_1}{q_1}\left(a_2(K_2) -\frac{p_2^2+q_2^2+1}{24q_2^2}\right) \\ 
    &\quad +l (a_2(K_1)+ a_2(K_2)) - a_3(L) + \frac{l^2}{12}(\frac{p_1}{q_1} + \frac{p_2}{q_2}) +\frac{l(l^2 -1)}{12} \Biggr) \\ 
    &\quad +|p_1p_2-q_1q_2l^2|\left(\frac{\sigma(E(\L))}{8} + \frac{s(p_1,q_1)}{2} + \frac{s(p_2,q_2)}{2} \right).  
\end{aligned}
\end{align}
Here, for $E(\L) = \begin{pmatrix} l_{11} & l_{12} \\ l_{21} & l_{22} \end{pmatrix} = \begin{pmatrix} p_1/q_1 & l \\ l & p_2/q_2 \end{pmatrix}$, as noted in Equation~\eqref{eq:theta}, the term $\theta(\L_{\{1,2\}})$ is calculated by 
\begin{align*}
\theta(\L_{\{1,2\}}) &= l_{12}l_{21}l_{11}+l_{11}l_{12}l_{21}+l_{21}l_{12}l_{22}+l_{22}l_{21}l_{12}+l_{12}l_{21}l_{12}+l_{21}l_{12}l_{21}-2l\\
&=2 l^2 p_1/q_1 + 2 l^2 p_2/q_2 + 2 l^3 -2l. 
\end{align*}
See~\cite[p.\ 19]{Lescop96} for details. 
Notice that, since $q_1, q_2 > 0$, 
\[
\sign(p_1p_2-q_1q_2l^2) = \sign\left(\dfrac{p_1p_2}{q_1q_2} -l^2\right), 
\]
and 
\begin{align*}
    \sign\left(\dfrac{p_1p_2}{q_1q_2} -l^2\right) &= \sign\left(\det E(\L) \right) \\ 
    &= \begin{cases}
        1 & \text{if } b_-(\L) = 0, 2, \\ 
        -1 & \text{if } b_-(\L) = 1
    \end{cases}\\ 
    &=(-1)^{b_-(\L)}. 
\end{align*}
Thus, we have $\sign(p_1p_2-q_1q_2l^2) = (-1)^{b_-(\L)}$. 
Then the RHS of Equation~\eqref{eq:LescopFor2comp} is equal to 
\begin{align*}
(-1)^{b_-(\L)} q_1 q_2\Biggl( \frac{p_2}{q_2}\left(a_2(K_1) -\frac{p_1^2+q_1^2+1}{24q_1^2} \right)+\frac{p_1}{q_1} \left(a_2(K_2) -\frac{p_2^2+q_2^2+1}{24q_2^2}\right) \\ 
    \quad +l (a_2(K_1)+ a_2(K_2)) - a_3(L) + \frac{l^2}{12} \left( \frac{p_1}{q_1} + \frac{p_2}{q_2} \right) +\frac{l(l^2 -1)}{12} \\ 
    \quad +\left(\frac{p_1p_2}{q_1q_2} - l^2\right)\left(\frac{\sigma(E(\L))}{8} + \frac{s(p_1,q_1)}{2} + \frac{s(p_2,q_2)}{2} \right)
    \Biggr) .  
\end{align*}
Then we can see that the RHS of Equation~\eqref{eq:Ito} equals the RHS of \eqref{eq:LescopFor2comp}.

\bibliographystyle{amsplain}
\bibliography{IJT2025}

\end{document}